\documentclass[11pt]{llncs}
\usepackage{amsmath,amssymb,mathtools}
\usepackage[all]{xy}
\usepackage{fullpage}
\newcommand*{\EnsembleQuotient}[2]%
{\ensuremath{%
    #1/\!\raisebox{-.65ex}{\ensuremath{#2}}}}

\DeclareMathAlphabet{\mathpzc}{OT1}{pzc}{m}{it}
\mathchardef\mhyphen="2D

\begin{document}
%
%
\pagestyle{headings}  
%
%
\title{Linear induction algebra and a normal form for linear operators}
\titlerunning{Linear induction algebra}  
%
\author{Laurent Poinsot}
\authorrunning{Laurent Poinsot} 
%
%
\institute{Universit\'e Paris 13, Sorbonne Paris Cit\'e, LIPN, CNRS (UMR 7030), France,\\
\email{laurent.poinsot@lipn.univ-paris13.fr},\\ 
\texttt{http://lipn.univ-paris13.fr/\~{}poinsot/}
}

\maketitle              

\begin{abstract}
The set of natural integers is fundamental for at least two reasons: it is the free induction algebra over the empty set (and at such allows definitions of maps by primitive recursion) and it is the free monoid over a one-element set, the latter structure being a consequence of the former. In this contribution, we study the corresponding structure in the linear setting, \emph{i.e.} in the category of modules over a commutative ring rather than in the category of sets, namely the free module generated by the integers. It also provides free structures of induction algebra and of monoid (in the category of modules). Moreover we prove that each of its linear endomorphisms admits a unique normal form, explicitly constructed, as a non-commutative formal power series.
\\
\quad\\
\textbf{Mathematics Subject Classification (2010) 03C05, 08B20, 18D35}
\end{abstract}

\section{Overview}
The set of natural integers is fundamental for at least two reasons: it is the free induction algebra over the empty set (and at such allows definitions of maps by primitive recursion) and it is the free monoid over a one-element set, the latter structure being a consequence of the former. It is possible to define a similar object, with similar properties, in the category of modules over some commutative ring $R$ (with a unit), namely the free $R$-module $V$ generated by $\mathbb{N}$. We prove that this module inherits from the integers a structure of initial $R$-linear induction algebra, and also of free $R$-linear monoid (a usual $R$-algebra). General definitions of varieties of algebraic structures (in the setting of universal algebra) in the category of $R$-modules, rather than set-based, are given in section~\ref{basic} together with some results concerning the relations between a set-theoretic algebra and its $R$-linear counterpart. These results are applied to $V$ in section~\ref{linear-induction-algebra}, and allow us to outline a theory of $R$-linear recursive functions, and to provide relations between the (free) monoid structure of $V$ and well-known usual algebraic constructions (polynomials, tensor algebra and algebra of a monoid). Finally in section~\ref{normal-form} we prove that any $R$-linear endomorphism of $V$ may be written uniquely as an infinite sum, and so admits a unique \emph{normal form} as a non-commutative formal power series. 

\section{Linear universal algebra}\label{basic}
In this contribution are assumed known some basic notions about category theory and universal algebra that may be found in any textbooks (\cite{Cohn,MacLane} for instance). We also refer to~\cite{BouAlg} for notions concerning modules and their tensor product. However some of them are recalled hereafter. The basic categories used are the category $\mathpzc{Set}$ of sets (with set-theoretic maps) and the category $R\mhyphen\mathpzc{Mod}$ of modules over some fixed commutative ring $R$ with a unit (and $R$-linear maps).   If $C$ denotes a category and $a,b$ are two objects of this category, then the class of all morphisms from $a$ to $b$ in $C$ is denoted by $C(a,b)$. For instance, if $V,W$ are two $R$-modules, then $R\mhyphen\mathpzc{Mod}(V,W)$ denotes the set of all $R$-linear maps from $V$ to $W$. Let $(\Sigma,\alpha)$ be a (finitary and homogeneous) signature (also called an algebra type or an operator domain), \emph{i.e.}, a set  $\Sigma$ (the elements of which are referred to as \emph{symbols} of functions) together with a map $\alpha\colon \Sigma\rightarrow \mathbb{N}$ called the \emph{arity function}. In what follows we simply denote by $\Sigma$ a signature $(\Sigma,\alpha)$, and $\alpha^{-1}(\{\,n\,\})$ is denoted by $\Sigma(n)$. The elements of $\Sigma(0)\subseteq\Sigma$ with an arity of zero  are called \emph{symbols of constants}. A \emph{$\Sigma$-algebra}, or algebra of type $\Sigma$, is a pair $(A,F)$ where $A$ is a set and $F$ is a map that associates to each symbol of function $f$ of arity $\alpha(f)=n$ (for each $n$) an actual map $F(f)\colon A^n\rightarrow A$ (we sometimes call $F$ the \emph{$\Sigma$-algebra structure map} of $A$). In particular if $\alpha(c)=0$, then $F(c)$ is identified to an element of $A$ (which explains the term of symbol of constant). An homomorphism between two algebras $(A,F),(B,G)$ over the same signature $\Sigma$ is a set-theoretic map $\phi\colon A\rightarrow B$ such that for every $f\in\Sigma(n)$, and every $a_1,\cdots,a_n\in A$, $\phi(F(f)(a_1,\cdots,a_n))=G(f)(\phi(a_1),\cdots,\phi(a_n))$ (in particular for each $c\in\Sigma(0)$, $\phi(F(c))=G(\phi(c))$). An \emph{isomorphism} is a homomorphism which is also a bijective map. A \emph{sub-algebra} $(B,G)$ of $(A,F)$ is a $\Sigma$-algebra such that the natural inclusion $B\subseteq A$ is a homomorphism of $\Sigma$-algebras. A \emph{congruence} $\cong$ on a $\Sigma$-algebra $(A,F)$ is an equivalence relation on $A$ such that for every $f\in\Sigma(n)$, if $a_i\cong b_i$, $i=1,\cdots,n$, then $F(f)(a_1,\cdots,a_n)\cong F(f)(b_1,\cdots,b_n)$. This implies that the quotient set $\EnsembleQuotient{A}{\cong}$ inherits a natural structure of $\Sigma$-algebra from that of $A$. It is well-known (see~\cite{Cohn}) that such congruences form a lattice, and then for every $R\subseteq A^2$, we may talk about the \emph{least congruence on $A$ generated by $R$} in an evident way. For any set $X$ there exists a \emph{free $\Sigma$-algebra} $\Sigma[X]$ on $X$. It is constructed by induction as follows (it is a subset of the free monoid $(\Sigma\sqcup X)^*$ over $\Sigma\sqcup X$, and the parentheses to form its elements are only used for readability; see~\cite{Cohn}). The base cases: $\Sigma(0)\subseteq \Sigma[X]$ and $X\subseteq \Sigma[X]$, the induction rule: for every $n$, and every $f\in \Sigma(n)$, if $t_1,\cdots,t_n\in \Sigma[X]$, then $f(t_1,\cdots,t_n)\in \Sigma[X]$, and the closure property: it is the least subset of $(\Sigma\sqcup X)^*$ with these two properties. Its structure of $\Sigma$-algebra is the evident one. It is called \emph{free} because for any $\Sigma$-algebra $(A,F)$ and any set-theoretic map $\phi\colon X\rightarrow A$, there exists a unique homomorphism $\widehat{\phi}\colon \Sigma[X]\rightarrow (A,F)$ such that $\widehat{\phi}(x)=\phi(x)$ for every $x\in X$. In category-theoretic terms, this means that the (obvious) forgetful functor from the category of $\Sigma$-algebras to $\mathpzc{Set}$ admits a left adjoint, and this implies that a free algebra is unique up to a unique isomorphism (we can talk about \emph{the} free algebra). 
\begin{example}
The set $\mathbb{N}$, together with the constant $0$ and the usual successor function, is the free induction algebra over the empty set (see for instance~\cite{Cohn}) where we call \emph{induction algebra} any algebra over the signature $\mathpzc{Ind}=\{\, 0,S\,\}$ where $0\in \mathpzc{Ind}(0)$ and $S\in \mathpzc{Ind}(1)$.
\end{example}
A \emph{variety} of $\Sigma$-algebras is a class of algebras closed under homomorphic images, sub-algebras, and direct products. A \emph{law} or \emph{identity} over  $\Sigma$ on the \emph{standard alphabet} $\mathpzc{X}=\{\, x_i\colon i\geq 0\,\}$ is a pair $(u,v)\in \Sigma[\mathpzc{X}]^2$ sometimes written as an equation $u=v$. We say that a law $(u,v)$ \emph{holds} in a $\Sigma$-algebra $(A,F)$, or that $(A,F)$ \emph{satisfies} $(u,v)$, if under every homomorphism $\Sigma[\mathpzc{X}]\rightarrow (A,F)$ the values of $u$ and $v$ coincide. If $E$ is any set of laws in $\Sigma[\mathpzc{X}]$, then $\mathpzc{V}_{\Sigma,E}$ or simply $\mathpzc{V}_{E}$, is the class of all algebras which satisfy all the laws in $E$. By the famous Garrett Birkhoff's theorem, $\mathpzc{V}_{E}$ is a variety and any variety arises in such a way. 
\begin{example}
The variety of all monoids is given by $\mathpzc{V}_{M,E}$ where $M(0)=\{\, 1\,\}$, $M(2)=\{\, \mu\,\}$, $M(n)=\emptyset$ for every $n\not=0,2$, and $E$ consists in the three equations $(\mu(x_1,1),x_1)$, $(\mu(1,x_1),x_1)$ and $(\mu(\mu(x_1,x_2),x_3),\mu(x_1,\mu(x_2,x_3)))$. 
\end{example}
A \emph{free} algebra over a set $X$ in a variety $\mathpzc{V}_{\Sigma,E}$ is a $\Sigma$-algebra $V_X$ in the class $\mathpzc{V}_{\Sigma,E}$, together with a set-theoretic map $i_X\colon X\rightarrow V_X$, such that for every algebra $(A,F)$ in $\mathpzc{V}_{\Sigma,E}$ and every map $\phi\colon X\rightarrow A$, there is a unique homomorphism $\widehat{\phi}\colon V_X\rightarrow (A,F)$ with $\widehat{\phi}\circ i_X=\phi$. Thus the free $\Sigma$-algebra $\Sigma[X]$ is easily seen as a free algebra in the variety $\mathpzc{V}_{\Sigma,\emptyset}$. In category-theoretic terms, when a variety is seen as a category (whose morphisms are the homomorphisms of algebras), this means that the obvious forgetful functor from the variety to $\mathpzc{Set}$ admits a left adjoint. This implies that a free algebra is unique up to a unique isomorphism. Let us see a way to construct it. Let $\cong_E$ be the least congruence of $\Sigma$-algebra on $\Sigma[X]$ generated by the relations $\{\, (\widehat{\sigma}(u),\widehat{\sigma}(v))\colon (u,v)\in E,\ \sigma\colon \mathpzc{X}\rightarrow \Sigma[X]\,\}$ (recall that $\widehat{\sigma}\colon \Sigma[\mathpzc{X}]\rightarrow \Sigma[X]$ is the unique homomorphism of $\Sigma$-algebras that extends $\sigma$). Let $V_X=\EnsembleQuotient{\Sigma[X]}{\cong_E}$ together with its structure of quotient $\Sigma$-algebra inherited from that of $\Sigma[X]$. Let $(B,G)$ be any $\Sigma$-algebra in the variety $\mathpzc{V}_{\Sigma,E}$, and $\phi\colon X\rightarrow B$ be a set-theoretic map. It admits a unique homomorphism extension $\widehat{\phi}\colon \Sigma[X]\rightarrow (B,G)$ since $\Sigma[X]$ is free. Because $(B,G)$ belongs to  $\mathpzc{V}_{\Sigma,E}$ and $\widehat{\phi}\circ \sigma\colon \Sigma[\mathpzc{X}]\rightarrow (B,G)$ is a homomorphism whenever $\sigma\colon \Sigma[\mathpzc{X}]\rightarrow \Sigma[X]$ is so, then for each $(u,v)\in E$, $\widehat{\phi}(\sigma(u))=\widehat{\phi}(\sigma(v))$. Therefore $\widehat{\phi}$ passes to the quotient by $\cong_E$ and defines a homomorphism from $V_X$ to $(B,G)$ as expected.
\begin{example}
For instance $\mathbb{N}$ with its structure of (commutative) monoid is the free algebra in $\mathpzc{V}_{M,E}$ over $\{\, 1\,\}$, while $\mathbb{N}\setminus\{\,0\,\}$ with its multiplicative structure of monoid in the free algebra in $\mathpzc{V}_{M,E}$ over the set of all prime numbers. 
\end{example}
Up to now, we only describe \emph{set-based} algebras. But it is possible to talk about \emph{linear algebras}. For this let us recall some basic facts about modules and their tensor product (see~\cite{BouAlg}). Let $X$ be any set. The free $R$-module generated by $X$ is the $R$-module $RX$ of all formal sums $\sum_{x\in X}\alpha_x x$ ($\alpha_x\in R$) where all but finitely many coefficients $\alpha_x\in R$ are zero (this is the free $R$-module with basis $X$), and for any $x_0\in X$, we refer to the element $e_{x_0}\in RX$, obtained as the formal sum $\sum_{x\in X}\alpha_x x$ with $\alpha_x=0$ for every $x\not=x_0$ and $\alpha_{x_0}=1$ (the unit of $R$), as the \emph{canonical image of $x_0$ into $RX$}, and therefore this defines a one-to-one map $e\colon X\rightarrow RX$ by $e(x)=e_x$. If $W$ is any over $R$-module, then any $R$-linear map $\phi\colon RX\rightarrow W$ is entirely defined by its values on the basis $X$. Let $V_1,\cdots,V_n,W$ be  $R$-modules. A map $\phi\colon V_1\times \cdots \times V_n\rightarrow W$ is said to be \emph{multilinear} (or \emph{bilinear} when $n=2$) if it is linear in each of its variables when the other ones are fixed. Given a multilinear map $\phi\colon V_1\times \cdots \times V_n\rightarrow W$, there is a unique \emph{linear map} $\psi\colon V_1\otimes_R\cdots \otimes_R V_n\rightarrow V$, where $\otimes_R$ denotes the tensor product over $R$ (see~\cite{BouAlg}), such that $\psi\circ q=\phi$ (where $q\colon V_1\times \cdots\times V_n\rightarrow V_1\otimes_R\cdots \otimes_R V_n$ is the canonical multilinear map; the image of $(v_1,\cdots,v_n)$ under $q$ is denoted by $v_1\otimes\cdots \otimes v_n$). In what follows, $\phi$ is referred to as the \emph{multilinear map} associated to $\psi$, and denoted by $\psi_0$. If $V_1, \cdots, V_n$ are free qua $R$-modules with basis $(e^{(j)})_{i\in I_j}$, $j=1,\cdots,n$, then $V_1\otimes_R \cdots \otimes_R V_n$ also is free with basis $\{\, e^{(1)}_{i_1}\otimes \cdots \otimes e^{(n)}_{i_n}\colon i_j\in I_j,\ j=1,\cdots,n\,\}$.  Moreover given a linear map $\phi\colon V_1\rightarrow R\mhyphen\mathpzc{Mod}(V_2,V_3)$, then it determines a unique linear map $\psi\colon V_1\otimes_R V_2\rightarrow V_3$ (it is obtained from the bilinear map $\phi^{\prime}\colon V_1\times V_2 \rightarrow V_3$ given by $\phi^{\prime}(v_1,v_2)=\phi(v_1)(v_2)$). 
\begin{lemma}\label{multi-to-linear}
For every sets $X_1,\cdots, X_n$, $R(X_1\times \cdots \times X_n)$ and $RX_1\otimes_R\cdots \otimes_R RX_n$ are isomorphic $R$-modules.
\end{lemma}
\begin{proof}(Sketch) 
It is clear that $R(X_1\times \cdots \times X_n)$ is identified as a sub-module of $R(RX_1\times \cdots \times RX_n)$ by $\iota\colon (x_1,\cdots,x_n)\mapsto e(e(x_1),\cdots,e(x_n))$. 
Let $q\circ \iota\colon R(X_1\times \cdots \times X_n)\rightarrow RX_1\otimes_R\cdots \otimes_R RX_n$ be the restriction of the canonical multilinear map (it is clearly onto and is easily shown to be $R$-linear), and $s\colon RX_1\times \cdots \times RX_n\rightarrow R(X_1\times \cdots \times X_n)$ be the multilinear map given by $s(e(x_1),\cdots,e(x_n))=e(x_1,\cdots,x_n)$ for every $x_i\in X_i$, $i=1,\cdots,n$. Therefore it gives rise to a unique linear map $\widetilde{s}\colon RX_1\otimes_R \cdots \otimes_R RX_n\rightarrow R(X_1\times \cdots \times X_n)$. It is easy to see that $\widetilde{s}\circ q=id$, but $q$ is onto so that it is an $R$-linear isomorphism (the details are left to the reader). \hfill $\qed$
\end{proof}
From lemma~\ref{multi-to-linear}, it follows that any set-theoretic map $\phi\colon X_1\times\cdots \times X_n\rightarrow W$ may be extended in a unique way to a linear map $\widetilde{\phi}\colon RX_1\otimes_R \cdots \otimes_R RX_n \rightarrow W$. Following the notations from the proof of lemma~\ref{multi-to-linear}, $\phi\colon X_1\times\cdots \times X_n\rightarrow W$ is first freely extended to a $R$-linear map $\phi\colon R(X_1\times \cdots\times X_n)\rightarrow W$, and then $\phi\circ \widehat{s}\colon RX_1\otimes_R \cdots \otimes_R RX_n\rightarrow W$ is the expected linear map $\widetilde{\phi}$. Moreover its associated multilinear map $\widetilde{\phi}_0\colon RX_1\times \cdots \times RX_n\rightarrow W$ is sometimes referred to as the \emph{extension of $\phi$ by multilinearity}. We are now in position to introduce $R$-linear $\Sigma$-algebras and varieties. Let $\Sigma$ be an operator domain, and $R$ be a commutative ring with a unit. A \emph{$R$-linear $\Sigma$-algebra} is a $R$-module given a structure of $\Sigma$-algebra such that all operations are $R$-multilinear. More precisely it is a $R$-module $V$ with a $\Sigma$-algebra structure map $F$ such that for each $f\in \Sigma(n)$ ($n\geq 0$), $F(f)\colon \underbrace{V\otimes_R \cdots \otimes_R V}_{n\ \mathit{factors}}\rightarrow V$ is $R$-linear. Following~\cite{Bergman}, if $V$ is a $R$-linear $\Sigma$-algebra, let $\mathcal{U}(V)$ denote its underlying (set-theoretic) $\Sigma$-algebra (its structure of $\Sigma$-algebra is given by the multilinear map $F_0(f)\colon V\times \cdots \times V \rightarrow V$ associated to $F(f)$), and if $(A,F)$ is a usual $\Sigma$-algebra, let $(RA,\widetilde{F})$ denote the $R$-linear $\Sigma$-algebra made from the free $R$-module $RA$ on $A$ by extending the $\Sigma$-operation $F(f)$, $f\in \Sigma(n)$, of $A$ by multilinearity. More precisely, $\widetilde{F}(f)\colon RA\otimes_R\cdots\otimes_R RA\rightarrow RA$ is the unique linear map obtained from lemma~\ref{multi-to-linear}. It is given by $\widetilde{F}(f)(e(a_1)\otimes\cdots\otimes e(a_n))=e(F(f)(a_1,\cdots,a_n))$ for each $a_1,\cdots,a_n\in A$ (this map is well defined since $\{\, e(a_1)\otimes \cdots \otimes e(a_n)\colon a_1,\cdots,a_n\in A\,\}$ forms a basis over $R$ of $RA$). (According to the above discussion, this is equivalent to a multilinear map $\widetilde{F}_0(f)\colon RA^n \rightarrow RA$ with $\widetilde{F}_0(f)(e(a_1,\cdots,a_n))=e(F(f)(a_1,\cdots,a_n))$.) Actually we obtain a \emph{functorial} correspondence between $\Sigma$-algebras and $R$-linear $\Sigma$-algebras: the forgetful functor $\mathcal{U}$ admits a left adjoint given by the construction $RA$. More precisely, given a $R$-linear $\Sigma$-algebra $(W,G)$, and a homomorphism  $\phi\colon (A,F)\rightarrow \mathcal{U}(W,G)$, $\widetilde{\phi}\colon (RA,\widetilde{F})\rightarrow (W,G)$, given by $\widetilde{\phi}(e(a))=\phi(a)$ for each $a\in A$, is the unique extension of $\phi$ which is a homomorphism of $R$-linear $\Sigma$-algebras (this means that $\widetilde{\phi}$ is $R$-linear, and $\widetilde{\phi}(\widetilde{F}(f)(x_1\otimes \cdots \otimes x_n))=G(f)(\widetilde{\phi}(x_1)\otimes \cdots\otimes \widetilde{\phi}(x_n))$ for every $x_1,\cdots,x_n\in RA$. To determine such a correspondence between varieties and linear varieties we must be more careful due to multilinearity. A law $u=v$ on $\mathpzc{X}$ is said to be \emph{regular} when the same elements of $\mathpzc{X}$ occur in $u$ and $v$, and exactly once in both of them. For instance, $\mu(x_1,1)=x_1$, $\mu(\mu(x_1,x_2),x_3)=\mu(x_1,\mu(x_2,x_3))$ are regular laws. If $E$ is any set of regular equations on $\Sigma[\mathpzc{X}]$, and $(V,F)$ is a $R$-linear $\Sigma$-algebra, then we say that $(V,F)$ \emph{satisfies} $E$ when under all homomorphisms $\Sigma[\mathpzc{X}]\rightarrow \mathcal{U}(V)$, the images of $u$ and of $v$ are equal for each $(u,v)\in E$. If $E$ is any set of regular equations on $\Sigma[\mathpzc{X}]$, then there is a very close connection between the variety $\mathpzc{V}_{\Sigma,E}$ of $\Sigma$-algebras satisfying $E$, and the variety $\mathpzc{V}_{\Sigma,R,E}$ of $R$-linear $\Sigma$-algebras satisfying $E$: it is easy to see that a $R$-linear $\Sigma$-algebra $V$ will lie in $\mathpzc{V}_{\Sigma,R,E}$ if, and only if, $\mathcal{U}(V)$ lies in $\mathpzc{V}_{\Sigma,R,E}$. Conversely, according to~\cite{Bergman}, a $\Sigma$-algebra $(A,F)$ will lie in $\mathpzc{V}_{\Sigma,E}$ if, and only if, $RA$ lies in $\mathpzc{V}_{\Sigma,R,E}$. A \emph{free} $R$-linear algebra in $\mathpzc{V}_{\Sigma,R,E}$ over a set (resp. a $R$-module, resp. a $\Sigma$-algebra in the variety $\mathpzc{V}_{\Sigma,E}$) $X$  is a $R$-linear $\Sigma$-algebra $V_X$ in the variety $\mathpzc{V}_{\Sigma,R,E}$ with a set-theoretic map (resp. a $R$-linear map, resp. a homomorphism) $j_X\colon X\rightarrow V_X$ (called the \emph{canonical map}) such that for all $R$-linear algebra $W$ in $\mathpzc{V}_{\Sigma,R,E}$ and all set-theoretic map (resp. $R$-linear map, resp. homomorphism) $\phi\colon X\rightarrow W$ there is a unique homomorphism $\widehat{\phi}\colon V_X\rightarrow W$ of $R$-linear algebras  such that $\widehat{\phi}\circ j_X=\phi$. Such a free algebra is unique up to a unique isomorphism. As an example, the free $R$-linear algebra in  $\mathpzc{V}_{\Sigma,R,E}$ over a $R$-module $W$ is made as follows. Let us assume that the free $R$-linear algebra $V_W$  on the underlying \emph{set} $W$ is constructed with the \emph{set-theoretic} map $j_W\colon W\rightarrow V_W$ (we see in lemma~\ref{freelinearalgebra} that it always exists). Let $F$ be the $\Sigma$-algebra structure map of $V_W$ (this means that $F(f)$ is a linear map from $V_W\otimes_R\cdots \otimes_R V_W\rightarrow V_W$ for each $f\in \Sigma$). Let $\overline{W}$ be the least sub-module of $V_W$ \emph{stable under all $F(f)$'s} (this means that the image of $\overline{W}\otimes_R \cdots\otimes_R \overline{W}$ by all $F(f)$'s lies into $\overline{W}$) and that contains the sub-module generated by $j_W(w_1+w_2)-j_W(w_1)-j_W(w_2)$, $j_W(\alpha w)-\alpha j_W(w)$ for every $\alpha\in R$, $w_1,w_2,w\in W$. Then it is easily seen that the quotient module $\EnsembleQuotient{V_W}{\overline{W}}$ inherits a structure of $R$-linear $\Sigma$-algebra from that of $V_W$, and is the expected free algebra (where the canonical map is the composition of the natural epimorphism $V_W\rightarrow \EnsembleQuotient{V_W}{\overline{W}}$ with the set-theoretic canonical map $W\rightarrow V_W$). 
\begin{remark}
These three notions of free algebras (over a set, a module or a $\Sigma$-algebra in the variety $\mathpzc{V}_{\Sigma,E}$) come from the fact that there are three forgetful functors, and each of them admits a left adjoint. 
\end{remark}
 
\begin{lemma}\label{freelinearalgebra}
Let $E$ be a set of regular equations. Let $X$ be a set and $V_X$ be the free $\Sigma$-algebra over $X$ in the variety $\mathpzc{V}_{\Sigma,E}$ with $i_X\colon X\rightarrow V_X$. Then, $RV_X$ with $j_X\colon X\rightarrow RV_X$ given by $j_X(x)=e(i_X(x))$  is the free $R$-linear $\Sigma$-algebra over $X$ in $\mathpzc{V}_{\Sigma,R,E}$. Moreover, $RV_X$, with the $R$-linear map $k_{RX}\colon RX\rightarrow RV_X$ defined by $h_{RX}(e_x)=e(i_X(x))$ for every $x\in X$, is the free $R$-linear $\Sigma$-algebra (in $\mathpzc{V}_{\Sigma,R,E}$) over $RX$. Finally, let $k_{V_X}\colon V_X\rightarrow RV_X$ be the unique homomorphism such that $k_{V_X}\circ i_X=j_X=e\circ i_X$. Then, $RV_X$ with $h_{V_X}$ is free over $V_X$. 
\end{lemma}
\begin{proof}
(The proof of this lemma is easy for a category theorist or universal algebraist but is given for the sake of completeness.) 
Let $(W,G)$ be a $R$-linear $\Sigma$-algebra, and $\phi\colon X\rightarrow W$ be a set-theoretic map. Then, there exists a unique homomorphism of $\Sigma$-algebras $\widehat{\phi}\colon V_X\rightarrow \mathcal{U}(W)$ such that $\widehat{\phi}\circ i_X=\phi$. Since $RV_X$ is free with basis $V_X$ over $R$, there is a unique $R$-linear map $\psi\colon RV_X\rightarrow W$ such that $\psi\circ e=\widehat{\phi}$ (so $\psi\circ j_X=\psi\circ e\circ j_X=\widehat{\phi}\circ j_X=\phi$). Moreover from  the above discussion we know that $RV_X$ is a $R$-linear $\Sigma$-algebra of the variety $\mathpzc{V}_{\Sigma,R,E}$. It remains to prove that $\psi$ is a homomorphism of $\Sigma$-algebra from $V_X$ to $W$. Let $f\in\Sigma(n)$, and $a_1,\cdots,a_n\in V_X$. Let $F$ be the $\Sigma$-algebra structure map of $V_X$. We have $\psi(\widetilde{F}(f)(e(a_1)\otimes \cdots \otimes e(a_n)))=\psi(e(F(f)(a_1,\cdots,a_n)))
=\widehat{\phi}(F(f)(a_1,\cdots,a_n))=G(f)(\widehat{\phi}(a_1),\cdots,\widehat{\phi}(a_n))
=G(f)(\psi(e(a_1))\otimes\cdots \otimes \psi(e(a_n)))$, for each $a_1,\cdots,a_n\in V_X$. Now, let $\phi\colon RX\rightarrow W$ be any $R$-linear map (where $W$ is a $R$-linear $\Sigma$-algebra in the variety $\mathpzc{V}_{\Sigma,R,E}$). Then, there exists a unique set-theoretic map $\phi_0\colon X\rightarrow W$ such that $\phi_0(x)=\phi(e(x))$ for every $x\in X$. Therefore there exists a unique homomorphism of $\Sigma$-algebras $\widehat{\phi}_0\colon V_X\rightarrow W$ such that $\widehat{\phi}_0\circ i_X=\phi_0$. Finally, there exists a unique $R$-linear map, wich is also a homomorphism of $\Sigma$-algebras $\psi\colon RV_X\rightarrow W$ such that $\psi\circ e=\widehat{\phi}_0$. Then, $\phi_0=\psi\circ j_X=\psi\circ e\circ i_X=\psi\circ h_{RX}\circ e$. But $\phi\circ e=\phi_0$, and both maps $\phi$ and $\psi\circ h_{RX}$ are $R$-linear and equal on basis elements of $RX$, so that they are equal on $RX$ as expected. Finally, let $\phi\colon V_X\rightarrow \mathcal{U}(W)$ be a homomorphism of $\Sigma$-algebras. Then, there exists a unique set-theoretic map $\phi_0\colon X\rightarrow W$ such that $\phi_0=\phi\circ i_X$. Then, there exists a unique homomorphism of $\Sigma$-algebras which is a $R$-linear map $\widehat{\phi}_0\colon RV_X\rightarrow W$ with $\widehat{\phi}_0\circ j_X=\phi_0$. Then, $\widehat{\phi}_0\circ k_{V_X}\circ i_X=\widehat{\phi}_0\circ j_X=\phi_0=\phi\circ i_X$, and since $\widehat{\phi}_0\circ k_{V_X}$ and $\phi$ are both homomorphisms from $V_X$ to $W$ it follows that their are equal (since $V_X$ is free). \hfill $\qed$
\end{proof}

\section{$R$-linear induction algebra}\label{linear-induction-algebra}

\subsection{The initial $R$-linear induction algebra}\label{initial}

The free $R$-module $R\mathbb{N}$ over $\mathbb{N}$ is denoted by $V$. The canonical image of an integer $n$ into $V$ is denoted by $e_n$ so $e_i\not=e_j$ for every $i\not=j$ and $\{\, e_n\colon n\in\mathbb{N}\,\}$ happens to be a basis of $V$ over $R$. The constant $0$ of the signature $\mathpzc{Ind}$ corresponds to $e_0$, and the successor map $s\colon \mathbb{N}\rightarrow \mathbb{N}$ is uniquely extended by $R$-linearity (no need here of multilinearity) to $U\in R\mhyphen\mathpzc{Mod}(V,V)$ defined on the basis elements by $Ue_n=e_{n+1}$, $n\in\mathbb{N}$. It is clear that $(V,e_0,U)$ is a $R$-linear induction algebra, and according to lemma~\ref{freelinearalgebra}, $(V,e_0,U)$ is even the free $R$-linear induction algebra over the empty set, the free $R$-linear induction over the zero vector space, and the free $R$-linear induction over the induction algebra $\mathbb{N}$. We call $(V,U,e_0)$ the \emph{initial} $R$-linear induction algebra because given another $R$-linear induction algebra $(W,w,S)$ ($w\in W$, $S\in R\mhyphen\mathpzc{Mod}(W,W)$), there is a unique $R$-linear map $\phi\colon V\rightarrow W$ such that $\phi(e_0)=w$, and $\phi\circ U=S\circ \phi$. This may be proved directly from the fact that $V$ is free over $(e_n)_{n\in\mathbb{N}}$, and $e_n=U^n(e_0)$ for each $n\in \mathbb{N}$. (Indeed, there is a unique linear map $\phi\colon V\rightarrow W$ such that $\phi(e_n)=S^n(w)$.) 
\begin{remark}
It is obvious that $\mathbb{N}$ is the initial induction algebra (since it is freely generated by the empty set). This means that for each induction algebra $A$, we have a natural isomorphism (see~\cite{MacLane} for a precise definition of this notion) of sets $\mathpzc{V}_{\mathpzc{Ind},\emptyset}(\mathbb{N},A)\cong\mathpzc{Set}(\emptyset,A)=\{\, \emptyset\,\}$ (where the variety  $\mathpzc{V}_{\mathpzc{Ind},\emptyset}$ of all induction algebras is considered as a category). Now, since $V$ is the free $R$-linear induction algebra on $\mathbb{N}$, for every $R$-linear induction algebra $W$, one also has  natural isomorphisms (of sets) $\mathpzc{V}_{\mathpzc{Ind},R,\emptyset}(V,W)\cong \mathpzc{V}_{\mathpzc{Ind},\emptyset}(\mathbb{N},\mathcal{U}(W))\cong \{\, \emptyset\,\}$. 
\end{remark}
For every $n\in\mathbb{N}$, let $V_n$ be the sub-module of $V$ generated by $(e_{k})_{k\geq n}$ (which is obviously free over $(e_{k})_{k\geq n}$). It is a $R$-linear induction algebra on its own $(V_n,e_n,U)$ (since $U\colon V_n\rightarrow V_{n+1}\subseteq V_n$). Therefore, for every $n\in\mathbb{N}$, there exists a unique $R$-linear map, which is a homomorphism of induction algebras, $\mu_n\colon V\rightarrow V_n$ such that $\mu_n(e_0)=e_n$ and $\mu_n(e_{k+1})=\mu_n(U e_k)=U(\mu_n(e_k))$. It is easy to prove by induction that $\mu_n(e_k)=e_{k+n}$. Now, we define $\overline{\mu}\colon V\rightarrow R\mhyphen\mathpzc{Mod}(V,V)$ by $\overline{\mu}(e_n)=\mu_n$ for each $n\geq 0$. Therefore we obtain a bilinear map $\mu_0\colon V\times V\rightarrow V$ given by $\phi(e_m,e_n)=\overline{\mu}(e_m)(e_n)=\mu_m(e_n)=e_{m+n}$. Finally this leads to the existence of a linear map $\mu\colon V\otimes_R V\rightarrow V$ defined by $\mu(e_m\otimes e_n)=e_{m+n}$. A simple calculation shows that $\mu$ is associative (in the sense that $\mu(\mu(u\otimes v)\otimes w)=\mu(u\otimes\mu(v\otimes w))$ for every $u,v,w\in V$ and not only for basis elements) and $\mu(v\otimes e_0)=v=\mu(e_0\otimes v)$ for every $v\in V$. This means that $V$ becomes a monoid, and more precisely an $R$-algebra (an internal monoid in the category of $R$-modules, see~\cite{MacLane}). We see below another way to build this $R$-algebra structure on $V$. 
\begin{remark}\label{irregular-laws}
Similarly it is also possible to define the free linear extension of the usual multiplication on $\mathbb{N}$ to a linear map $\mu ^{\prime}\colon V\otimes V\rightarrow$ by $\mu^{\prime}(e_m\otimes e_n)=e_{mn}$, which happens to be associative and has a unit $e_1$. But $e_0$ is not an absorbing element: for instance $\mu^{\prime}((\alpha e_m + \beta e_n)\otimes e_0)=\alpha\mu^{\prime}(e_m\otimes e_0)+\beta\mu ^{\prime}(e_n\otimes e_0)=(\alpha+\beta)e_0\not=e_0$ whenever $\alpha+\beta\not=0$ (in $R$). It is due to the fact that the equation $x_1\times 0=0$ or $0\times x_1=0$ is not a regular law. Similarly, even if we have $\mu^{\prime}(e_m\otimes\mu(e_n\otimes e_p))=\mu(\mu^{\prime}(e_m\otimes e_n)\otimes \mu^{\prime}(e_m\otimes e_p))$, the distributivity law does not hold for any $u,v,w\in V$ (again essentially because it is not a regular law). 
\end{remark}
\subsection{A free monoid structure and its links with classical algebra}

We also know that $(\mathbb{N},+,0)$ is the free monoid $\{\,1\,\}^*$ over $\{\,1\,\}$. Therefore, again by lemma~\ref{freelinearalgebra}, $(V,\mu,e_0)$ is the free monoid over $\{\, 1\,\}$, or over the module $R$, or over the monoid $(\mathbb{N},+,0)$, where $\mu\colon V\otimes_R V\rightarrow V$ is the $R$-linear map given by $\mu(e_m\otimes e_n)=e_{m+n}$ (it satisfies $\mu(\mu(u\otimes v)\otimes w)=\mu(u\otimes\mu(v\otimes w))$ for every $u,v,w\in V$, and $\mu(v\otimes e_0)=v=\mu(e_0\otimes v)$ for every $v\in V$). Therefore $(V,\mu,e_0)$ has a structure of commutative $R$-algebra which is actually the same as that defined in subsection~\ref{initial}. Moreover it is nothing else than the usual algebra of polynomials $R[x]$ in one indeterminate (an isomorphism is given by $e_n\mapsto x^n$). The fact that $(V,\mu,e_0)$ is the free monoid over $R$ is also re-captured by the fact that $R[x]$ may be seen as the tensor $R$-algebra generated by $Rx\cong R$ (see~\cite{BouAlg}). Finally the fact that $(V,\mu,e_0)$ is free over the monoid $(\mathbb{N},+,0)$ is recovered in the usual algebraic setting by the fact that qua a $R$-algebra $V$ (and therefore $R[x]$) is isomorphic to the $R$-algebra of the monoid $\mathbb{N}$. 
\begin{remark}
According to the remark~\ref{irregular-laws}, there is no hope to use the multiplication from $\mathbb{N}$ in order to define a structure of ring on $V$ internal to the category of modules.
\end{remark}
\subsection{Linear primitive recursion operator}

Back to the fact that $V$ is the initial $R$-linear induction algebra, we show here how to define linear maps by primitive recursion in a way similar to the usual clone of primitive recursive functions (see for instance~\cite{Soare}). Recall that given two maps $g\colon \mathbb{N}^k\rightarrow \mathbb{N}$ and $h\colon \mathbb{N}^{k+2}\rightarrow \mathbb{N}$ it is possible to define a unique map $R(g,h)=f\colon \mathbb{N}^{k+1}\rightarrow \mathbb{N}$ by \emph{primitive recursion} as $f(0,n_1,\cdots,n_k)=g(n_1,\cdots,n_k)$ and $f(n+1,n_1,\cdots,n_k)=h(n_1,\cdots,n_k,n,f(n_1,\cdots,n_k))$ for every $n_1,\cdots,n_k,n\in\mathbb{N}$.  
If $W$ is a $R$-module, then $W^{\otimes n}$ is the tensor product $\underbrace{W\otimes_R \cdots\otimes_R W}_{n\ \mathit{times}}$ (so that $W^{\otimes 0}\cong R$). Now, any set-theoretic map $f\colon \mathbb{N}^{\ell}\rightarrow \mathcal{U}(V)$ gives rise to a unique $R$-linear map $\widehat{f}\colon V^{\otimes \ell}\rightarrow V$ by $\widehat{f}(e_{n_1}\otimes \cdots \otimes e_{n_{\ell}})=f(n_1,\cdots,n_{\ell})$. Therefore given $g\colon \mathbb{N}^k\rightarrow V$ and $h\colon \mathbb{N}^{k+2}\rightarrow V$, there exists a unique $R$-linear map $\widehat{R}(g,h)\colon V^{\otimes k+1}\rightarrow V$ by $\widehat{R}(g,h)(e_{n_1}\otimes \cdots\otimes e_{n_{k+1}})=R(g,h)(n_1,\cdots, n_{k+1})$ and thus by $\widehat{R}(g,h)(e_0\otimes e_{n_1}\otimes \cdots \otimes e_{n_k})=g(n_1,\cdots,n_k)=\widehat{g}(e_{n_1}\otimes\cdots \otimes e_{n_k})$ and 
$\widehat{R}(g,h)(e_{n+1}\otimes e_{n_1}\otimes \cdots \otimes e_{n_k})=h(n_1,\cdots,n_k,n,R(g,h)(n,n_1,\cdots,n_k))
=\widehat{h}(e_{n_1}\otimes\cdots\otimes e_{n_k}\otimes e_n\otimes \widehat{R}(g,h)(e_n\otimes e_{n_1}\otimes \cdots \otimes e_{n_k}))$
for each $n,n_1,\cdots,n_k,n_{k+1}\in\mathbb{N}$. The following result is then proved.
\begin{theorem}[Linear primitive recursion]
Let $g\in \mathpzc{Set}(\mathbb{N}^k,V)$, $h\in \mathpzc{Set}(\mathbb{N}^{k+2},V)$. Then there exists a unique linear map $\phi\colon V^{\otimes k+1}\rightarrow V$ such that  $\phi(e_0\otimes e_{n_1}\otimes \cdots\otimes e_{n_k})=\widehat{g}(e_{n_1}\otimes \cdots\otimes e_{n_k})$ and $\phi(e_{n+1}\otimes e_{n_1}\otimes\cdots\otimes e_{_k})=\widehat{h}(e_{n_1}\otimes\cdots\otimes e_{n_k}\otimes e_n\otimes \phi(e_n\otimes e_{n_1}\otimes \cdots \otimes e_{n_k}))$ for every $n,{n_1},\cdots, {n_k}\in \mathbb{N}$.
\end{theorem}
\begin{remark}
The two $R$-linear maps $\mu$ and $\mu^{\prime}$ from subsection~\ref{initial} may be obtained by linear primitive recursion.
\end{remark}
In order to close this subsection, let us briefly see the corresponding notion of clone of primitive recursive functions in the linear case. 
Let $f\colon R\mhyphen\mathpzc{Mod}(V^{\otimes m},V)$, and $g_1\cdots,g_m\in R\mhyphen\mathpzc{Mod}(V^{\otimes n},V)$, then the \emph{superposition} $\mu(f,g_1,\cdots,g_m)$ in $R\mhyphen\mathpzc{Mod}(V^{\otimes n},V)$ is defined by 
$\mu(f,g_1,\cdots,g_m)(e_{i_1}\otimes \cdots \otimes e_{i_n})=f(g_1(e_{i_1}\otimes \cdots \otimes e_{i_n})\otimes \cdots \otimes g_n(e_{i_1}\otimes \cdots \otimes e_{i_n}))$  
for every $e_{i_1},\cdots, e_{i_n}\in V$. For every $n$, $i=1,\cdots,n$, we define the \emph{projections} $\pi_i^{(n)}\in R\mhyphen\mathpzc{Mod}(V^{\otimes n},V)$ by $\pi_i^{(n)}(e_{j_1}\otimes \cdots \otimes e_{j_n})=e_{j_i}$ for every $j_1,\cdots,j_n\in \mathbb{N}$. Then the clone of all linear primitive recursive functions is the set of all $R$-linear maps from $V^{\otimes k}$, for varying $k$, to $V$ which is closed under superposition, and linear primitive recursion (in the sense that if $g\colon V^{\otimes k}\rightarrow V$ and $h\colon V^{\otimes k+2}\rightarrow V$ are primitive recursion linear maps, then $\widehat{R}(g_0,h_0)$ is linear primitive recursive, where $g_0\colon \mathbb{N}^k\rightarrow V$ and $h_0\colon \mathbb{N}^{k+2}\rightarrow V$ are the unique maps such that $g=\widehat{g}_0$ and $h=\widehat{h}_0$), that contains, for every set-theoretic primitive recursive function $f\in\mathbb{N}^k\rightarrow \mathbb{N}$, the map $\widetilde{f}\colon V^{\otimes k}\rightarrow V$ where $\widetilde{f}(e_{i_1}\otimes \cdots \otimes e_{i_k})=e_{f(i_1,\cdots,i_k)}$ for all $(i_1,\cdots,i_k)\in\mathbb{N}^k$, and that contains the projections. 

\section{A normal form for $R$-linear endomorphisms of $V$}\label{normal-form}

In~\cite{Duchamp} the authors generalize a result from~\cite{Kurbanov} that concerns the decomposition of linear endomorphisms of $V$ (in~\cite{Duchamp} only the case where $R$ is a field is considered) with respect to  a pair of raising and  lowering ladder operators. In the present paper, after recalling this result in a more general setting, we show that it may be seen as a strong version of Jacobson's density theorem and that it gives rise to a unique normal form for the endomorphisms of $V$ in a way made precise hereafter. 

\subsection{Jacobson's density theorem}

Jacobson's density theorem is a result made of two parts: an algebraic and a topological one. Let us begin with definitions needed for the algebraic part. Let $R$ be a unitary ring (commutative or not). If $M$ is a left $R$-module, then we denote by $\nu\colon R\rightarrow \mathpzc{Ab}(M)$ the associated (module) structure map (where $\mathpzc{Ab}$ denotes the category of all Abelian groups). This is a ring map since it is a linear representation of $R$. A left $R$-module $M$ is said to be a \emph{faithful module} if the structure map is one-to-one, \emph{i.e.}, $\ker \nu=(0)$. A left $R$-module $M$ is said to be a \emph{simple module} if it is non-zero and it has no non-trivial sub-modules (modules different from $(0)$ and $M$ itself). Finally, a ring $R$ is said to be \emph{(left-)primitive} if it admits a faithful simple left-module. Now, let us turn to the topological part. Given two topological spaces $X,Y$, we let $\mathpzc{Top}(X,Y)$ be the set of all continuous maps from $X$ to $Y$ (here $\mathpzc{Top}$ denotes the category of all topological spaces). Let $K$ be a compact subset of $X$ and $U$ be an open set in $Y$, then we define $V(K,U)=\{\, f\in \mathpzc{Top}(X,Y)\colon f(K)\subseteq U\,\}$. 
The collection of all such sets $V(K,U)$ (with varying $K$ and $U$) forms a subbasis for the \emph{compact-open topology} on $\mathpzc{Top}(X,Y)$. This means that for every non-void open set $V$ in the compact-open topology, and every $f\in V$, there exist compact sets $K_1,\cdots,K_n$ of $X$ and open sets $U_1,\cdots,U_n$ in $Y$ such that $f\in \bigcap_{i=1}^n V(K_i,U_i)\subseteq V$, see~\cite{Arens1,Eilenberg}.
\begin{remark}
Let $R$ be a ring (commutative or not), and let $M$ be a left module over $R$. Let us assume that $M$ has the discrete topology. Therefore its compact subsets are exactly its finite subsets. Then, the compact-open topology induced by $\mathpzc{Top}(M,M)=M^M$ on the sub-space of all $R$-linear endomorphisms $R\mhyphen\mathpzc{Mod}(M,M)$ of $M$ is the same as the topology of simple convergence (here $R\mhyphen\mathpzc{Mod}$ is the category of all left $R$-modules), \emph{i.e.} for every topological space $X$, a map $\phi\colon X\rightarrow R\mhyphen\mathpzc{Mod}(M,M)$ is continuous if, and only if, for every $v\in M$, the map $\phi_v\colon x\in X\mapsto \phi(x)(v)\in M$ is continuous. Moreover with this topology, and $R$ discrete, $R\mhyphen\mathpzc{Mod}(M,M)$ is a Hausdorff complete topological $R$-algebra (\cite{Warner}).
\end{remark}
We are now in position to state Jacobson's density theorem (see~\cite{Farb} for a proof). 
\begin{theorem}[Jacobson's density theorem]
Let $R$ be a unitary ring (commutative or not). The ring $R$ is primitive if, and only if, it is a dense subring (in the compact-open topology) of a ring $\mathbb{D}\mhyphen\mathpzc{Mod}(M,M)$ of linear endomorphisms of some (left) vector space $M$ over a division ring $\mathbb{D}$ (where $M$ is discrete). 
\end{theorem}

\subsection{Decomposition of endomorphisms}\label{decomp}

A direct consequence of Jacobson's density theorem is the following. Let $\mathbb{K}$ be a field of characteristic zero, and $A(\mathbb{K})$ be the \emph{Weyl algebra} which is the quotient algebra of the free algebra $\mathbb{K}\langle x,y\rangle$ in two non-commutative variables by the two-sided ideal generated by $xy-yx-1$ (this means that although the generators of $A(\mathbb{K})$ do not commute  their commutator is equal to $1$). (See~\cite{Connell} for more details.) Now, $A(\mathbb{K})$ is a primitive ring by Jacobson's density theorem. Indeed, $A(\mathbb{K})$ admits a faithful representation into $\mathbb{K}\mhyphen\mathpzc{Mod}(\mathbb{K}[z],\mathbb{K}[z])$ by $[x]\mapsto (P(z)\mapsto zP(z))$ and $[y]\mapsto (P(z)\mapsto \frac{d}{dz}P(z))$ (where $P(z)$ denotes an element of $\mathbb{K}[z]$,  $[x],[y]$ are the canonical images of $x,y$ onto $A(\mathbb{K})$, and it is clear that the commutation relation is preserved by this representation), and it is an easy exercise to check that through this representation $A(\mathbb{K})$ is a dense subring of $\mathbb{K}\mhyphen\mathpzc{Mod}(\mathbb{K}[z],\mathbb{K}[z])$ (under the topology of simple convergence with $\mathbb{K}[z]$ discrete). Nevertheless given  $\phi\in \mathbb{K}\mhyphen\mathpzc{Mod}(\mathbb{K}[z],\mathbb{K}[z])$ and an open neighborhood $V$ of $\phi$, Jacobson's density theorem does not provide any effective nor even constructive way to build some $\phi_0\in A(\mathbb{K})$ such that $\phi_0\in V$.  In~\cite{Kurbanov} the authors show how to build in a recursive way a sequence of operators $(\Omega_n)_{n\in\mathbb{N}}$, $\Omega_n\in A(\mathbb{K})$ for each $n$, such that $\lim_{n\rightarrow\infty}\Omega_n=\phi$. In~\cite{Duchamp} the authors generalize this result to the case of $\mathbb{K}$-linear endomorphisms of $V$, with $\mathbb{K}$ any field (of any characteristic), proving that the multiplicative structure of the algebra $\mathbb{K}[z]$ is unnecessary (recall that as $\mathbb{K}$-vector spaces, $V\cong \mathbb{K}[z]$). We now recall this result in a more general setting where a commutative ring $R$ with unit replaces the field $\mathbb{K}$. Let $(e_n)_{n\in\mathbb{N}}$ be a basis of $V=R\mathbb{N}$. We define a $R$-linear map $D\colon V\rightarrow V$ by $D(e_0)=0$ and $D(e_{n+1})=e_n$ for every $n\in \mathbb{N}$. (This linear map $D$ may be given a definition by linear primitive recursion as $D=\widehat{R}(0,\pi_2^{(1)})$.) According to~\cite{Poinsot} (see page 109), for any sequence $(\phi_n)_{n\in \mathbb{N}}$ with $\phi_n\in R\mhyphen\mathpzc{Mod}(V,V)$, the family $(\phi_n \circ D^n)_{n\in\mathbb{N}}$ is summable in the topology of simple convergence of $R\mhyphen\mathpzc{Mod}(V,V)$ (where, for every endomorphism $\phi$ of $V$, $\phi^0=id_V$ and $\phi^{n+1}=\phi\circ \phi^n$). This means that there is an element of $R\mhyphen\mathpzc{Mod}(V,V)$ denoted by $\sum_{n\in\mathbb{N}}\phi_n \circ D^n$, and called the \emph{sum} of $(\phi_n\circ D^n)_{n\geq 0}$, such that for every $v\in V$, $v\not=0$, $\left (\sum_{n\in\mathbb{N}}\phi_n\circ D^n\right )(v)=\sum_{n=0}^{d(v)}\phi(D^n(v))$, where $d(v)$ is the maximum of all $k$'s such that the coefficient of $e_k$ in the decomposition of $v$ in the basis $(e_n)_{n\geq 0}$ is non-zero.
\begin{remark}
The above summability of $(\phi_n \circ D^n)_{n\in\mathbb{N}}$ essentially comes from topological nilpotence of $D$ in the topology of simple convergence which means that for every $v\in V$, there exists $n_v\in \mathbb{N}$ (for instance $d(v)$ when $v\not=0$) such that for every $n\geq n_v$, $D^n(v)=0$ ($D^n \rightarrow 0$ in the topology of simple convergence).
\end{remark}
For every polynomial $P(x)=\sum_{n=0}^m p_n x^n\in R[x]$, every sequence $\mathbf{v}=(v_n)_{n\in\mathbb{N}}$ of elements of $V$, and every $R$-linear endomorphism $\phi$ of $V$, we define $P(\mathbf{v})=\sum_{n=0}^m p_n v_n\in V$, and $P(\phi)=\sum_{n=0}^m p_n \phi^n\in R\mhyphen\mathpzc{Mod}(V,V)$. It is clear that $P(x)\in R[x]\mapsto P(\mathbf{e})\in V$ for $\mathbf{e}=(e_n)_{n\in\mathbb{N}}$ defines a linear isomorphism between $R[x]$ and $V$. Moreover we have $P(\mathbf{e})=P(U)(e_0)$. Now, let $\phi$ be given. There exists a sequence of polynomials $(P_n(x))_n$ such that $\phi=\sum_{n\in\mathbb{N}}P_n(U)\circ D^n$ (this means that $\phi$ is the sum of the summable family $(P_n(U)\circ D^n)_{n\geq 0}$ and it is equivalent to $\phi(e_n)=\sum_{k=0}^nP_k(U)(D^k(e_n))$ for each $n\in\mathbb{N}$, because $D^k(e_n)=0$ for every $k>n$). This can be proved by induction on $n$ as follows. We have $\phi(e_0)=P_0(\mathbf{e})=P_0(U)(e_0)$ for a unique $P_0(x)\in R[x]$. Let us assume that there are $P_1(x),\cdots,P_n(x)\in R[x]$ such that $\phi(e_n)=\sum_{k=0}^n P_k(U)D^k(e_n)=\sum_{k=0}^n P_k(U)e_{n-k}$. Let $P_{n+1}(U)(e_0)=P_{n+1}(\mathbf{e})=\phi(e_{n+1})-\sum_{k=0}^n P_k(U)e_{n+1-k}$ ($P_{n+1}$ is uniquely determined). Then, $\phi(e_{n+1})=\sum_{k=0}^{n+1}P_k(U)\circ D^k (e_{n+1})$. 
\begin{remark}
This result is outside the scope of Jacobson's density theorem since $R$ is not a division ring, and also more precise since it provides a recursive algorithm to construct explicitly a sequence that converges to any given endomorphism.
\end{remark}
Every sequence $(P_n)_n$ defines an endomorphism $\phi$ given by the sum of $(P_n(U)\circ D^n)_n$, and the above construction applied to $\phi$ recovers the  sequence $(P_n)_n$. The correspondence between $\phi$ and $(P_n)_n$ as constructed above is functional, and it is actually  a $R$-linear map ($R[x]^{\mathbb{N}}$ is the product $R$-module), onto and one-to-one. 

\subsection{A normal form for $R$-linear endomorphisms of $V$}\label{normalform}

Let us consider the following subset of the $R$-algebra of non-commutative series $R\langle\langle x,y\rangle\rangle$ in two variables (see~\cite{BouAlg}): $R\langle x,y\rangle\rangle=\{\, \sum_{n\geq 0}P_n(x)y^n\colon \forall n,\ P_n(x)\in R[x]\,\}$. This is a $R$-sub-module of $R\langle\langle x,y\rangle\rangle$, and a $R[x]$-module with action given by $Q(x)\cdot (\sum_{n\geq 0}P_n(x)y^n)=\sum_{n\geq 0}(Q(x)P_n(x))y^n=(\sum_{n\geq 0}P_n(x)y^n)\cdot Q(x)$. (We observe that $xy=y\cdot x$ but $yx$ does not belong to $R\langle x,y\rangle\rangle$. ) According to the result of subsection~\ref{decomp}, there exists a $R$-linear isomorphism $\pi\colon R\langle x,y\rangle\rangle \rightarrow R\mhyphen\mathpzc{Mod}(V,V)$ which maps $\sum_{n\geq 0}P_n(x)y^n$ to $\sum_{n\in\mathbb{N}}P_n(U)\circ D^n$.
\begin{remark}
It is essential that $xy\not=yx$, otherwise $\pi(xy)=U\circ D\not=id_V=D\circ U=\pi(yx)$, and $\pi$ would be ill-defined. 
\end{remark}
For any $\phi\in R\mhyphen\mathpzc{Mod}(V,V)$, the unique $S=\sum_{n\geq 0}P_n(x)y^n\in R\langle x,y\rangle\rangle$ such that $\pi(S)=\phi$ should be called the \emph{normal form} $s(\phi)$ of $\phi$ for a reason made clear hereafter. We observe that any set-theoretic map $\phi\colon \mathbb{N}\rightarrow \mathbb{N}$ also has such a normal form through the natural isomorphism $\mathpzc{Set}(\mathbb{N},\mathbb{N})\cong R\mhyphen\mathpzc{Mod}(V,V)$.  
\begin{example}
\begin{enumerate}
\item Let us assume that $R$ contains $\mathbb{Q}$ as a sub-ring. Let us consider the \emph{formal integration} operator $\int$ on $V$ defined by $\int e_n=\frac{e_{n+1}}{n+1}$ for every integer $n$. Then, $s(\int)=\sum_{n\geq 0}(-1)^{n}\frac{x^{n+1}}{(n+1)!}y^{n}$ (by recurrence).
\item Since the commutator $[D,U]=D\circ U-U\circ D=id_V-U\circ D$, we obtain $s([D,U])=1-xy$. 
\end{enumerate}
\end{example}
Let $\pi_0(x)=U$, $\pi_0(y)=D$,  and $\widehat{\pi}\colon \{\, x,y\,\}^*\rightarrow R\mhyphen\mathpzc{Mod}(V,V)$ be the unique monoid homomorphism extension of $\pi_0$ (where $R\mhyphen\mathpzc{Mod}(V,V)$ is seen as a monoid under composition). 
Let $R\{\{x,y\}\}$ be the set of all series $S=\sum_{w\in \{\, x,y\,\}^*}\alpha_w w$ in $R\langle \langle x,y\rangle\rangle$ such that the family $(\alpha_w \widehat{\pi}(w))_{w\in \{\, x,y\,\}^*}$ of endomorphisms of $V$ is summable. 
\begin{example}
Let us consider the series $S=\sum_{n\geq 0}y^nx^n\in R\langle\langle x,y\rangle\rangle$. Then, $S\not\in R\{\{x,y\}\}$ since $\widehat{\pi}(y^nx^n)=\pi_0(y)^n\circ\pi_0(x)^n=D^n\circ U^n=id_V$ for each $n$. Whereas $S^{\prime}=\sum_{n\geq 0}x^ny^n\in R\{\{x,y\}\}$ since $\sum_{n\geq 0}U^nD^n$ is equal to the operator $e_n\mapsto (n+1)e_n$. 
\end{example}
From general properties of summability~\cite{Warner}, $R\{\{x,y\}\}$ is a sub $R$-algebra of $R\langle\langle x,y\rangle\rangle$, and the homomorphism of monoids $\widehat{\pi}$ may be extended to an algebra map $\widetilde{\pi}\colon R\{\{x,y\}\}\rightarrow R\mhyphen\mathpzc{Mod}(V,V)$ by $\widetilde{\pi}(\sum_{w}\alpha_w w)=\sum_{w}\alpha_w \widehat{\pi}(w)$ which is obviously onto, so that $R\mhyphen\mathpzc{Mod}(V,V)\cong \EnsembleQuotient{R\{\{x,y\}\}}{\ker \widetilde{\pi}}$ (as $R$-algebras). We have $\widetilde{\pi}(s(\phi))=\phi$, so that $s$ defines a linear section of $\widetilde{\pi}$.  Let $\mathpzc{N}\colon R\{\{x,y\}\}\rightarrow R\langle x,y\rangle\rangle$ be the $R$-linear map defined by $\mathpzc{N}(S)=s(\widetilde{\pi}(S))$. Then, for every $S,S^{\prime}\in R\{\{x,y\}\}$, $S\cong {S^{\prime}}\bmod{\ker\widetilde{\pi}}$ (\emph{i.e.}, $\widetilde{\pi}(S)=\widetilde{\pi}(S^{\prime})$) if, and only if, $\mathpzc{N}(S)=\mathpzc{N}(S^{\prime})$. Also it holds that $\mathpzc{N}(\mathpzc{N}(S))=S$. The module of all normal forms $R\langle x,y\rangle\rangle$ inherits a structure of $R$-algebra by $S*S^{\prime}=\mathpzc{N}(SS^{\prime})=s(\widetilde{\pi}(SS^{\prime}))=s(\widetilde{\pi}(S)\circ \widetilde{\pi}(S^{\prime}))$ isomorphic to $R\mhyphen\mathpzc{Mod}(V,V)\cong \EnsembleQuotient{R\{\{x,y\}\}}{\ker \widetilde{\pi}}$.  
\begin{example}
We have $y*x=1$ while $x*y=xy$, so that $[y,x]=y*x-x*y=1-xy$. Let us define the operator $\partial$ on $V$ by $\partial e_{n+1}=(n+1)e_n$ for each integer $n$ and $\partial e_0=0$. Then, we have $s(\partial)=\sum_{n\geq 1}x^{n}y^{n-1}$. Moreover, $[\partial,U]=\partial\circ U-U\circ \partial=id_V$. It follows that $[s(\partial),x]=1$.  Let $A(R)$ be the quotient algebra $R\langle x,y\rangle$ by the two-sided ideal generated by $xy-yx-1$, namely the \emph{Weyl algebra over $R$}. Therefore there exists a unique morphism of algebras $\phi\colon A(R)\rightarrow R\langle x,y\rangle\rangle$ such that $\phi(x)=x$ and $\phi(y)=s(\partial)$. Composing with  the isomorphism $\pi\colon R\langle x,y\rangle\rangle\rightarrow R\mhyphen\mathpzc{Mod}(V,V)$, we obtain a representation of the algebra $A(R)$ on the module $V$ ($x$ acts on $V$ as $U$ while $y$ acts on $V$ as $\partial$). When $R$ is a field $\mathbb{K}$ of characteristic zero, then this representation is faithful (see~\cite{Bjork}), hence in this case $\mathbb{K}\langle x,y\rangle\rangle$ contains a copy of the Weyl algebra $A(\mathbb{K})$, namely the sub-algebra generated by $x$ and $s(\partial)$. 
\end{example}

\section{Concluding remarks and perspectives}

\subsection{Free linear induction algebras}

Let $X$ be any set. According to section~\ref{basic}, we may define the free $R$-linear induction algebra $V_X$ on $X$. It is isomorphic to the direct product of $|X|+1$ copies of $V$, namely the $R$-module $V\oplus \bigoplus_{x\in X} V$, this is so because the free induction algebra on $X$ is $\{ S^n(0)\colon n\geq 0\}\sqcup\bigsqcup_{x\in X}\{ S^n(x)\colon n\geq 0\}$ (where $\bigsqcup$ is the set-theoretic disjoint sum). As an example, take $X$ finite of cardinal say $n$, then $V_X$ is isomorphic to $V^{n+1}$. In this finite case, we have $R\mhyphen\mathpzc{Mod}(V^{n+1},V^{n+1})\cong R\mhyphen\mathpzc{Mod}(V,V)^{(n+1)^2}\cong R\langle x,y\rangle\rangle^{(n+1)^2}$. From subsection~\ref{normalform} it follows that any endomorphism of $V^{n+1}$ may be written as a vector of length $(n+1)^2$ or better a $(n+1)\times (n+1)$ matrix with entries some members of $R\langle x,y\rangle\rangle$. More generally, for each integers $m,n$, we have $R\mhyphen\mathpzc{Mod}(V^m,V^n)\cong R\mhyphen\mathpzc{Mod}(V,V)^{mn}$ so that we have obtained a complete description of all linear maps between spaces of the form $V^n$ in terms of the basic operators $U$ and $D$. 

\subsection{Links with Sheffer sequences}
It is not difficult to check that we may define a new associative multiplication on $R\langle x,y\rangle\rangle$, and therefore also on $R\mhyphen\mathpzc{Mod}(V,V)$, by $$\left(\sum_{n\geq 0}P_n(x)y^n\right)\#\left(\sum_{n\geq 0}Q_n(x)y^n\right)=\sum_{n\geq 0}\left(\sum_{k\geq 0}\langle P_n(x)\mid x^k \rangle Q_k(x)\right)y^n$$ where $\langle P(x)\mid x^k\rangle$ denotes the coefficient of $x^k$ in the polynomial $P(x)$ (so that in the above formula the sum indexed by $k$ is actually a sum with a finite number of non-zero terms for each $n$), with a two-sided identity $\sum_{n\geq 0}x^ny^n$ (that corresponds to the operator $e_n\mapsto (n+1)e_n$ of $V$). This product is a generalization of the so-called umbral composition~\cite{Roman}. Let us assume that $\mathbb{K}$ is a field of characteristic zero. 
Following~\cite{Rota} (see also~\cite{Poinsot2013}) a sequence  $(p_n(x))_{n\geq 0}$ of polynomials in $\mathbb{K}[x]$ such that the degree of $p_n(x)$ is $n$ for each integer $n$ is called a \emph{Sheffer sequence} if there are two series $\mu(y),\sigma(y)\in \mathbb{K}[[y]]$, where $x$ and $y$ are assumed to be commuting variables, with $\mu(0)\not=0$, $\sigma(0)=0$, and $\sigma^{\prime}(0)\not=0$ (where $\sigma^{\prime}$ denotes the usual derivation of series) such that $\sum_{n\geq 0}p_n(x)\frac{y^n}{n!}=\mu(y)e^{x\sigma(y)}\in\mathbb{K}[[x,y]]$. A series $S=\sum_{n\geq 0}\frac{1}{n!}p_n(x)y^n\in\mathbb{K}\langle x,y\rangle\rangle$ is said to be a \emph{Sheffer series} whenever $(p_n(x))_n$ is a Sheffer sequence. Such series correspond to \emph{Sheffer operators} on $V$ given by $\sum_{n\geq 0}\frac{1}{n!}p_n(U)\circ D^n$. For instance Laguerre's polynomials given by $L_n(x)=\sum_{k=0}^n\binom{n}{k}\frac{(-1)^k}{k!}x^k$ form a Sheffer sequence, and thus $\sum_{n\geq 0}\frac{1}{n!}L_n(U)\circ D^n$ is a Sheffer operator. We observe that the above multiplication $\#$ corresponds to the umbral composition of $(P_n(x))_n$ and $(Q_n(x))_n$. Because Sheffer sequences form a group under umbral composition (see~\cite{Roman}), it follows that Sheffer operators and Sheffer series form an isomorphic group under the corresponding umbral composition. The perspectives of our present contribution concern the study of such operators and their combinatorial properties. 

%
%

\end{document}